\documentclass[12pt,a4paper]{amsart}

\usepackage{amsrefs}

\title[Representation infinite quiver settings]{A characterization of representation infinite quiver settings}

\author{Grzegorz Bobi\'nski}

\address{Faculty of Mathematics and Computer Science \\ Nicolaus Copernicus University \\ ul. Chopina 12/18 \\ 87-100 Toru\'n \\ Poland}

\email{gregbob@mat.umk.pl}

\usepackage[all]{xy}

\newcommand{\bbA}{\mathbb{A}}
\newcommand{\bbD}{\mathbb{D}}
\newcommand{\bbE}{\mathbb{E}}
\newcommand{\bbM}{\mathbb{M}}
\newcommand{\bbN}{\mathbb{N}}
\newcommand{\bbZ}{\mathbb{Z}}

\newcommand{\bd}{\mathbf{d}}
\newcommand{\be}{\mathbf{e}}
\newcommand{\bh}{\mathbf{h}}
\newcommand{\bw}{\mathbf{w}}

\newcommand{\calC}{\mathcal{C}}
\newcommand{\calI}{\mathcal{I}}
\newcommand{\calL}{\mathcal{L}}

\newcounter{claim}
\newtheorem{coro}[claim]{Corollary}

\newtheorem{prop}[claim]{Proposition}
\newtheorem{theo}[claim]{Theorem}

\theoremstyle{remark}

\DeclareMathOperator{\GL}{GL}
\DeclareMathOperator{\rad}{rad}
\DeclareMathOperator{\rep}{rep}

\newcommand{\ol}{\overline}

\allowdisplaybreaks

\begin{document}

\begin{abstract}
We characterize pairs $(Q, \bd)$ consisting of a quiver $Q$ and a dimension vector $\bd$, such that over a given algebraically closed field $k$ there are infinitely many representations of $Q$ of dimension vector $\bd$. We also present an application of this result to the study of algebras with finitely many orbits with respect to the action of (the double product) of the group of units.
\end{abstract}

\maketitle

\section{Introduction}

Throughout the paper $k$ is an algebraically closed field. A well-known theorem of Gabriel~\cite{Gabriel} states that the Euclidean quivers are the minimal representation infinite quivers, where a quiver is called representation infinite, if it has infinitely many indecomposable $k$-representations (up to isomorphism). One also shows that $Q$ is representation infinite if and only if there exists a dimension vector $\bd$ such that there are infinitely many representations of $Q$ of dimension vector $\bd$ (a stronger and more general version of this observation is a content of the famous second Brauer--Thrall conjecture, first proved by Bautista~\cite{Bautista}). The main result of the paper can be viewed as a refinement of Gabriel's theorem. Namely, we show that the pairs $(Q, \bh_Q)$, where $Q$ is a Euclidean quiver and $\bh_Q$ is the associated radical vector, are the minimal pairs $(Q', \bd)$ consisting of a quiver $Q'$ and a dimension vector $\bd$, such that there are infinitely many representations of $Q'$ of dimension vector $\bd$ (we call such pairs representation infinite quiver settings).

The proof of the result consists of two steps, which seem to be known before, but apparently have not been joined together. First, we use a result of Skowro\'nski and Zwara~\cite{SkowronskiZwara} (see Proposition~\ref{prop criterion}), which says that a pair $(Q, \bd)$ is a representation infinite quiver setting if and only if there exists a nonzero dimension vector $\bd' \leq \bd$, such that $q_Q (\bd') \leq 0$, where $q_Q$ is the associated Tits form. The second step is Proposition~\ref{prop reverse}, which states that if $q_Q (\bd) \leq 0$, for a nonzero dimension vector $\bd$, then there exists a Euclidean subquiver $Q'$ of $Q$ such that $\bh_{Q'} \leq \bd'$, where $\bd'$ is the restriction of $\bd$ to $Q'$. This latter result should be well-known (at least to the experts), however we could not spot it in the literature. In particular, it seems that it has never been used in the context of quiver representations. Consequently, we include its short proof for completeness.

The problem discussed in the paper has a geometric aspect. Namely, a pair $(Q, \bd)$ is a representation finite (i.e., not representation infinite) quiver setting if and only if there are only finitely many orbits (with respect to a natural action) in the variety $\rep_Q (\bd)$ of representations of $Q$ of dimension vector $\bd$. In particular, there is a dense orbit in $\rep_Q (\bd)$, hence for example Schofield's description of the ring of semi-invariants~\cite{Schofield} applies. Moreover, according to~\cite{SkowronskiZwara}*{Theorem~2} the degeneration order in $\rep_Q (\bd)$ coincides with the extension order.

Another source of interest in the problem comes from the study of algebras with finitely many orbits (see~\cite{OkninskiRenner}). Here we say that an algebra $A$ has finitely many orbits if there are only finitely many orbits in $A$ with respect to the action of $U (A) \times U (A)$ given by $(u, v) \ast a := u a v^{-1}$, where $U (A)$ is the group of units of $A$ (see Section~\ref{sect algebras} for some motivation for this problem). As we explain in Section~\ref{sect algebras}, if $\rad^2 (A) = 0$, then one may associate to $A$ a quiver setting $(\Delta_A, \bw_A)$ in such a way, that $A$ has finitely many orbits if and only if $(\Delta_A, \bw_A)$ is representation finite. Consequently, our main theorem gives a criterion for $A$ to have finitely many orbits. We compare our criterion with~\cite{OkninskiRenner}*{Theorem~10} at the end of the paper.

The author acknowledges the support of the National Science Center grant no.\ 2015/17/B/ST1/01731.

\section{Main result}

By a \emph{quiver} $Q$ we mean a finite set $Q_0$ of \emph{vertices} together with a finite set $Q_1$ of \emph{arrows} and two maps $s = s_Q, t = t_Q \colon Q_1 \to Q_0$, which assign to each arrow $\alpha \in Q_1$ its \emph{starting vertex} $s \alpha$ and \emph{terminating vertex} $t \alpha$. An arrow with the same starting and terminating vertex is called a \emph{loop}. By a walk in a quiver $Q$ we mean a sequence $(x_0, \ldots, x_n)$ of vertices such that, for each $1 \leq i \leq n$, $x_{i - 1}$ and $x_i$ are connected by an arrow (i.e.\ there exists an arrow $\alpha \in Q_1$ such that $\{ s \alpha, t \alpha \} = \{ x_{i - 1}, x_i \}$). A quiver $Q$ is \emph{connected} if for any vertices $x, y \in Q_0$ there exists a walk $(x_0, \ldots, x_n)$ in $Q$ such that $x_0 = x$ and $x_n = y$. A walk $(x_0, \ldots, x_n)$ in $Q$ is called a \emph{cycle} if $n > 0$, $x_{i - 1} \neq x_{i + 1}$ for each $0 < i < n$, and $x_0 = x_n$ (in particular, $n \neq 2$). We say that a quiver $Q$ has \emph{multiple arrows} if there exist arrows $\alpha \neq \beta$ which connect the same vertices, i.e.\ $\{ s \alpha, t \alpha \} = \{ s \beta, t \beta \}$.

By a \emph{dimension vector} for a quiver $Q$ we mean an element of $\bbN^{Q_0}$. A \emph{representation} of $Q$ of dimension vector $\bd$ is a tuple $M = (M_\alpha)_{\alpha \in Q_1}$ of linear maps $M_\alpha \colon k^{\bd (s \alpha)} \to k^{\bd (t \alpha)}$, $\alpha \in Q_1$. The set of such representations is an affine space and we denote it by $\rep_Q (\bd)$. Let $\GL (\bd)$ be the set of tuples $g = (g_x)_{x \in Q_0}$ such that $g_x \colon k^{\bd (x)} \to k^{\bd (x)}$ is a $k$-linear automorphism for each $x \in Q_0$. Two representations $M$ and $N$ of dimension vector $\bd$ are said to be \emph{isomorphic} if the exists a tuple $g \in \GL (\bd)$ such that $N_\alpha = g_{t \alpha} M_\alpha g_{s \alpha}^{-1}$ for each $\alpha \in Q_1$. Following~\cite{Bocklandt} we call a pair $(Q, \bd)$ consisting of a quiver $Q$ and a dimension vector $\bd$ a \emph{quiver setting}. A quiver setting $(Q, \bd)$ is called \emph{representation finite} if there are only finitely many (up to isomorphism) representations of $Q$ of dimension vector $\bd$. Equivalently, there are finitely many $\GL (\bd)$-orbits in $\rep_Q (\bd)$ (with respect to the action induced by the isomorphism formula). A quiver setting which is not representation finite is said to be \emph{representation infinite}.

Let $Q$ and $Q'$ be two quivers. By a \emph{quiver morphism} $\psi \colon Q \to Q'$ we mean a pair of functions $\psi_0 \colon Q_0 \to Q_0'$ and $\psi_1 \colon Q_1 \to Q_1'$ such that $s_{Q'} (\psi_1 \alpha) = \psi_0 (s_Q \alpha)$ and $t_{Q'} (\psi_1 \alpha) = \psi_0 (t_Q \alpha)$. A quiver morphism $\psi \colon Q \to Q'$ is called \emph{injective} if both $\psi_0$ and $\psi_1$ are injective. If this is the case, then $Q$ may be viewed as a subquiver of $Q'$ (if we identify $Q$ with its image via $\psi$).

If $\psi \colon Q \to Q'$ is a quiver morphism and $\bd$ is a dimension vector for $Q$, then one defines a dimension vector $\psi_* \bd$ for $Q'$ via $(\psi_* \bd) (y) := \sum_{x \in \psi_0^{-1} (y)} \bd (x)$, for $y \in Q_0'$. Dually, if $\bd'$ is the dimension vector for $Q'$, then $\psi^* \bd'$ is a dimension vector for $Q$ given by $(\psi^* \bd') (x) :=  \bd' (\psi_0 x)$, for $x \in Q_0$.

Given two quiver settings $(Q, \bd)$ and $(Q', \bd')$ we write $(Q, \bd) \leq (Q', \bd')$, if there exists an injective quiver morphism $\psi \colon Q \to Q'$ such that $\bd \leq \psi^* \bd'$. Since $\psi$ is injective, the inequality $\bd \leq \psi^* \bd'$ is equivalent to the inequality $\psi_* \bd \leq \bd'$. Thus the condition $(Q, \bd) \leq (Q', \bd')$ means that $Q$ can be identified with a subquiver of $Q'$ in such a way that $\bd \leq \bd'|_{Q_0'}$ (with respect to this identification). Obviously, the relation $\leq$ is only a preorder on the class of quiver settings.

Recall that $Q$ is a \emph{Euclidean quiver} if its \emph{underlying graph} (the graph obtained from $Q$ by forgetting the orientations of arrows) is one of the following diagrams:
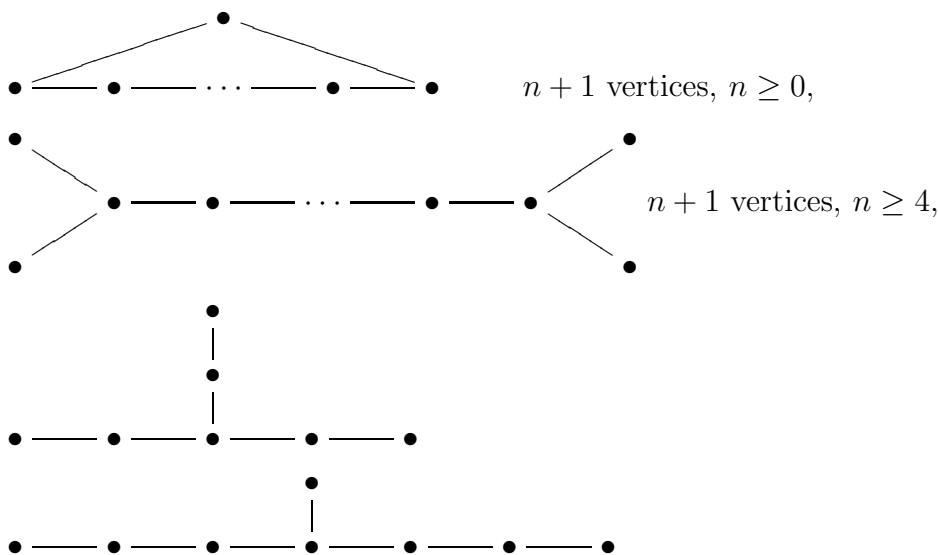
\begin{align*}
& \xymatrix@R=1em{& & \bullet \ar@{-}[rrd] \ar@{-}[lld] \\ \bullet \ar@{-}[r] & \bullet \ar@{-}[r] & \cdots & \bullet \ar@{-}[l] & \bullet \ar@{-}[l] &
\text{$n + 1$ vertices, $n \geq 0$,}}
\\
& \vcenter{\xymatrix@R=1em{\bullet \ar@{-}[rd] & & & & & & \bullet \ar@{-}[ld] \\ & \bullet \ar@{-}[r] & \bullet \ar@{-}[r] & \cdots & \bullet \ar@{-}[l] & \bullet \ar@{-}[l] & \\ \bullet \ar@{-}[ru] & & & & & & \bullet \ar@{-}[lu]}} \text{$n + 1$ vertices, $n \geq 4$,}
\\
& \xymatrix@R=1em{& & \bullet \ar@{-}[d] \\ & & \bullet \ar@{-}[d] \\ \bullet \ar@{-}[r] & \bullet \ar@{-}[r] & \bullet & \bullet \ar@{-}[l] & \bullet \ar@{-}[l]}
\\
& \xymatrix@R=1em{& & & \bullet \ar@{-}[d] \\ \bullet \ar@{-}[r] & \bullet \ar@{-}[r] & \bullet \ar@{-}[r] & \bullet & \bullet \ar@{-}[l] & \bullet \ar@{-}[l] & \bullet \ar@{-}[l]}
\\
& \xymatrix@R=1em{& & \bullet \ar@{-}[d] \\ \bullet \ar@{-}[r] & \bullet \ar@{-}[r] & \bullet & \bullet \ar@{-}[l] & \bullet \ar@{-}[l] & \bullet \ar@{-}[l] & \bullet \ar@{-}[l] & \bullet \ar@{-}[l]}
\end{align*}
In the above cases we say that $Q$ is a Euclidean quiver of type $\tilde{\bbA}_n$ ($\tilde{\bbA}$ for short), $\tilde{\bbD}_n$ ($\tilde{\bbD}$ shortly), $\tilde{\bbE}_6$, $\tilde{\bbE}_7$, and $\tilde{\bbE}_8$, respectively.

For each Euclidean quiver $Q$ there is a distinguished dimension vector $\bh_Q$ defined as follows:
\[
\bh_Q :=
\begin{cases}
\begin{smallmatrix}
& & 1 \\ 1 & 1 & \cdots & 1 & 1
\end{smallmatrix}
& \text{if $Q$ is of type $\tilde{\bbA}$},
\\
\begin{smallmatrix}
\genfrac{}{}{0pt}{1}{1}{1} & 2 & 2 & \cdots & 2 & 2 & \genfrac{}{}{0pt}{1}{1}{1}
\end{smallmatrix}
& \text{if $Q$ is of type $\tilde{\bbD}$},
\\
\begin{smallmatrix}
& & 1 & \\ & & 2 & \\ 1 & 2 & 3 & 2 & 1
\end{smallmatrix}
& \text{if $Q$ is of type $\tilde{\bbE}_6$},
\\
\begin{smallmatrix}
& & & 2 \\ 1 & 2 & 3 & 4 & 3 & 2 & 1
\end{smallmatrix}
& \text{if $Q$ is of type $\tilde{\bbE}_7$},
\\
\begin{smallmatrix}
& & 3 \\ 2 & 4 & 6 & 5 & 4 & 3 & 2 & 1
\end{smallmatrix}
& \text{if $Q$ is of type $\tilde{\bbE}_8$}.
\end{cases}
\]
The main theorem of the paper is the following.

\begin{theo} \label{theo main}
A quiver setting $(Q, \bd)$ is representation infinite if and only if there exists a Euclidean quiver $Q'$ such that $(Q', \bh_Q') \leq (Q, \bd)$.
In other words, $(Q, \bd)$ is representation infinite if and only if $Q$ has a Euclidean subquiver $Q'$ such that $\bh_{Q'} \leq \bd |_{Q_0'}$.
\end{theo}

We have the following reformulation of Theorem~\ref{theo main}, which can be seen as a refinement of Gabriel's Theorem.

\begin{coro}
The quiver settings $(Q, \bh_Q)$, where $Q$ is a Euclidean quiver, are precisely the minimal representation infinite quiver settings.
\end{coro}

An important role in the proof of the above theorem is played by the Euler form $q_Q$, which is the quadratic form $q_Q \colon \bbZ^{Q_0} \to \bbZ$ defined by
\[
q_Q (\bd) := \sum_{x \in Q_0} \bd (x)^2 - \sum_{\alpha \in Q_1} \bd (s \alpha) \bd (t \alpha) \qquad (\bd \in \bbZ^{Q_0}).
\]
The following fact is proved in~\cite{SkowronskiZwara}*{Section~3}.

\begin{prop} \label{prop criterion}
A quiver setting $(Q, \bd)$ is representation infinite if and only there exists a nonzero dimension vector $\bd'$ for $Q$ such that $\bd' \leq \bd$ and $q_Q (\bd') \leq 0$. \qed
\end{prop}

\begin{proof}[Proof of Theorem~\ref{theo main}, Part~I]
We first prove that if there exists a Euclidean quiver $Q'$ such that $(Q', \bh_Q') \leq (Q, \bd)$, then $(Q, \bd)$ is representation infinite. In fact, it is enough to prove that $(Q, \bh_Q)$ is representation infinite for each Euclidean quiver $Q$. This follows from well-known representation theory of Euclidean quivers (see for example~\cite{Ringel}*{Section~3.6}), but we may also refer to Proposition~\ref{prop criterion}, since $q_Q (\bh_Q) = 0$ as one easily checks.
\end{proof}

Before we continue the proof, we need some more notation. Let $(-, -) = (-, -)_Q \colon \bbZ^{Q_0} \times \bbZ^{Q_0} \to \bbZ$ be the symmetric bilinear form associated with $q_Q$, i.e.
\[
(\bd, \bd') := q_Q (\bd + \bd') - q_Q (\bd) - q_Q (\bd') \qquad (\bd, \bd' \in \bbZ^{Q_0}).
\]
In particular,
\[
q_Q (\bd) = \tfrac{1}{2} (\bd, \bd) = \tfrac{1}{2} \sum_{x \in Q_0} \bd (x) (\bd, \be_x),
\]
where $\be_x$, $x \in Q_0$, are the standard basis vectors.

We thank Daniel Simson for a hint, which allowed to significantly shorten the proof of the next result. In particular, the proof of the equality $q_Q (\bd) = 0$ below follows arguments in the proof of~\cite{SimsonSkowronski}*{Theorem~XIV.1.3}.

\begin{prop} \label{prop reverse}
Let $(Q, \bd)$ be a quiver setting such that $\bd$ is nonzero and $q_Q (\bd) \leq 0$. Then there exists a Euclidean subquiver $Q'$ of $Q$ such that $\bh_{Q'} \leq \bd |_{Q_0'}$.
\end{prop}

\begin{proof}
Obviously we may assume that $Q$ is connected and $\bd$ is \emph{sincere}, i.e.\ $\bd (x) \neq 0$ for each $x \in Q_0$. If there are either loops or multiple arrows in $Q$, then one easily sees that there is a Euclidean subquiver $Q'$ of $Q$ of type $\tilde{\bbA}$ such that $\bh_{Q'} \leq \bd |_{Q_0'}$. Thus for the rest of the proof we assume that $Q$ has neither loops nor multiple arrows. In particular, this implies that $|Q_0| \geq 3$. Indeed, if $Q_0 = \{ x \}$, then $q_Q (\bd) = \bd (x)^2 > 0$, while if $Q_0 = \{ x, y \}$, then
\[
q_Q (\bd) = \bd (x)^2 - \bd (x) \bd (y) + \bd (y)^2 = \tfrac{1}{2} [\bd (x)^2 + (\bd (x) - \bd (y))^2 + \bd (y)^2] > 0.
\]

We observe that we may assume $(\bd, \be_x) \leq q_Q (\bd)$ for each $x \in Q_0$. Indeed, if $(\bd, \be_x) \geq q_Q (\bd) + 1$ for some $x \in Q_0$, then
\[
(\bd - \be_x, \bd - \be_x) = q_Q (\bd) - (\bd, \be_x) + q_Q (\be_x) \leq q_Q (\bd) - (q_Q (\bd) + 1) + 1 \leq 0,
\]
hence we may replace $\bd$ by $\bd - \be_x$ and proceed by induction. Using the above assumption we get
\[
q_Q (\bd) = \tfrac{1}{2} \sum_{x \in Q_0} \bd (x) (\bd, \be_x) \leq \tfrac{1}{2} q_Q (\bd) \sum_{x \in Q_0} \bd (x).
\]
Consequently, if $q_Q (\bd) < 0$, then $\sum_{x \in Q_0} \bd (x) \leq 2$, which contradicts the inequality $|Q_0| \geq 3$, since $\sum_{x \in Q_0} \bd (x) \geq |Q_0|$. We conclude $q_Q (\bd) = 0$. Thus,
\[
0 = q_Q (\bd) \geq (\bd, \be_x) \geq \sum_{y \in Q_0} \bd (y) (\bd, \be_y) = 2 q_Q (\bd) = 0,
\]
i.e., $(\bd, \be_x) = 0$ for each $x \in Q_0$. From the famous Vinberg's characterization of Euclidean quivers (see for example~\cite{Kac}*{Corollary~4.3}), this implies that $Q$ is a Euclidean quiver. Moreover, another well-known property of Euclidean quivers (see for example~\cite{AssemSimsonSkowronski}*{Lemma~VII.4.2}) implies that $\bd = m \bh_Q$, for some $m > 0$, thus in particular $\bd \geq \bh_Q$.
\end{proof}

We finish now the proof of Theorem~\ref{theo main}.

\begin{proof}[Proof of Theorem~\ref{theo main}, Part~II]
Assume that $(Q, \bd)$ is a representation infinite quiver setting. We know from Proposition~\ref{prop criterion} that there exists a nonzero dimension vector $\bd'$ for $Q$ such that $\bd' \leq \bd$ and $q_Q (\bd') \leq 0$. Consequently, Proposition~\ref{prop reverse} implies that there exists a Euclidean subquiver $Q'$ of $Q$ such that $\bh_{Q'} \leq \bd' |_{Q_0'} \leq \bd |_{Q_0'}$.
\end{proof}

\section{An application to the algebras with finitely many orbits} \label{sect algebras}

For a $k$-algebra $A$ we denote by $U (A)$ the group of units of $A$. By abuse of notation (and following~\cite{OkninskiRenner}), by an \emph{$U (A)$-orbit in $A$} we mean an orbit with respect to the action of $U (A) \times U (A)$ on $A$ given by $(u, v) \ast a := u a v^{-1}$.  In~\cite{OkninskiRenner} the authors study the algebras with finitely many $U (A)$-orbits. In order to put this study in a wider context, we introduce some notation.

First, for an algebra $A$, $\calI (A)$ and $\calL (A)$ denote the lattices of ideals and left ideals, respectively. The group $U (A)$ acts on $\calL (A)$ by right translations and we denote by $\calC (A)$ the orbit space. Then $\calC (A)$ is a semigroup, with the product given by $[I] [J] := [I J]$, where $I J$ is the linear subspace of $A$ spanned by the products $a b$, $a \in I$, $b \in J$. The following theorem is one source of motivation for studying algebras with finitely many orbits.

\begin{theo}[\cite{OkninskiRenner}*{Theorem~6}] \label{theoOR}
Consider the following conditions for a finite dimensional $k$-algebra $A$.
\begin{enumerate}

\item
$A$ is of finite representation type.

\item
$\calC (A)$ is finite.

\item
$A$ has finitely many $U (A)$-orbits.

\item
$\calI (A)$ is a distributive lattice.

\end{enumerate}
Then $(1) \implies (2) \implies (3) \implies (4)$.
\end{theo}

Recall that an algebra $A$ is of \emph{finite representation type}, if there are only finitely many indecomposable $A$-modules (up to isomorphism). We remark that the implication $(3) \implies (4)$ holds, since we are working over an algebraically closed, hence infinite, field.

In~\cite{OkninskiRenner}*{Section~3} the authors study the algebras with finitely many orbits, such that $\rad^2 (A) = 0$, where $\rad (A)$ is the Jacobson radical of $A$. We discuss now a connection of this case with quiver settings. Let $A$ be a finite dimensional algebra with $\rad^2 (A) = 0$. Put $\ol{A} := A / \rad (A)$. By Wedderburn--Artin Theorem there exist positive integers $n_1$, \ldots, $n_l$ such that
\[
\ol{A} \simeq \bbM_{n_1} (k) \times \cdots \times \bbM_{n_l} (k).
\]
Let $\ol{e}_1$, \ldots, $\ol{e}_l$ be the idempotents corresponding to this decomposition, and $e_1$, \ldots, $e_l$ their lifts. For $1 \leq i \leq l$, we put $\ol{A}_i := \ol{e}_i \ol{A} \ol{e}_i$.

We associate to $A$ a quiver setting $(\Delta_A, \bw_A)$ in the following way. The vertices of $\Delta_A$ are the pairs $(0, i)$, $(1, i)$, for $1 \leq i \leq l$. There are only arrows of the form $(1, j) \to (0, i)$ and the number of arrows from $(1, j)$ to $(0, i)$ equals the rank $r_{i, j}$ of $e_i \rad (A) e_j$ as an $\ol{A}_i$-$\ol{A}_j$-bimodule. Finally, $\bw_A (0, i) := n_i =: \bw_A (1, i)$, for each $1 \leq i \leq l$.

We have the following result.

\begin{prop} \label{prop orbits}
Let $A$ be a finite dimensional algebra with $\rad^2 (A) = 0$. Then $A$ has finitely many orbits if and only if $(\Delta_A, \bw_A)$ is a representation finite quiver setting. Thus $A$ has finitely many $U (A)$-orbits if and only if there is no subquiver $Q$ of $\Delta_A$ of Euclidean type such that $\bh_Q \leq \bw_A |_{Q_0}$.
\end{prop}

\begin{proof}
Recall first from~\cite{OkninskiRenner}*{Proposition~9} that $A$ has finitely many $U (A)$-orbits if and only if there are finitely many $U (A) \times U (A)$-orbits in $\rad (A)$. Moreover, the $U (A) \times U (A)$-orbits in $\rad (A)$ coincide with the $U (\ol{A}) \times U (\ol{A})$-orbits in $\rad (A)$. Observe that $e_i \rad (A) e_j \simeq (\bbM_{n_i \times n_j} (k))^{r_{i, j}}$, thus
\[
\rad (A) = \prod_{1 \leq i, j \leq l} (\bbM_{n_i \times n_j} (k))^{r_{i, j}} = \rep_{\Delta_A} (\bw_A).
\]
Moreover,
\[
U (\ol{A}) \times U (\ol{A}) = \prod_{i = 1}^n \GL_{n_i} (k) \times \prod_{i = 1}^n \GL_{n_i} (k) = \GL (\bw_A),
\]
and the action of $U (\ol{A}) \times U (\ol{A})$ on $\rad (A)$ coincides with the action of $\GL (\bw_A)$ on $\rep_{\Delta_A} (\bw_A)$. This finishes the proof of the former statement. The latter one follows from the former one and Theorem~\ref{theo main}.
\end{proof}

Note that if $(\Delta, \bw) = (\Delta_A, \bw_A)$, for an algebra $A$, then we have the following:
\begin{enumerate}

\item
the vertex set $\Delta_0$ is a disjoint union of two subsets $\Delta_0'$ and $\Delta_0''$ such that there is a bijection $\delta \colon \Delta_0'' \to \Delta_0'$;

\item
every arrow starts in $\Delta_0''$ and terminates in $\Delta_0'$;

\item
$\bw (\delta y) = \bw (y)$, for each $x \in \Delta_0''$.

\end{enumerate}
Observe that every setting $(\Delta, \bw)$ with the above properties is (up to isomorphism) of the form $(\Delta_A, \bw_A)$, for some algebra $A$. Namely, as a vector space $A$ equals $\prod_{x \in \Delta_0'} \bbM_{\bw (x)} (k) \times \prod_{\substack{x \in \Delta_0' \\ y \in \Delta_0''}} (\bbM_{\bw (x) \times \bw (y)} (k))^{r_{x, y}}$, where $r_{x, y}$ is the number of arrows from $y$ to $x$, and the multiplication is given by
\begin{multline*}
((M_x)_{x \in \Delta_0'}, (M_{x, y, i})_{\substack{x \in \Delta_0' \\ y \in \Delta_0'' \\ 1 \leq i \leq r_{x, y}}}) ((N_x)_{x \in \Delta_0'}, (N_{x, y, i})_{\substack{x \in \Delta_0' \\ y \in \Delta_0'' \\ 1 \leq i \leq r_{x, y}}})
\\
:= ((M_x N_x)_{x \in \Delta_0'}, (M_x N_{x, y, i} + M_{x, y, i} N_{\delta (y)})_{\substack{x \in \Delta_0' \\ y \in \Delta_0'' \\ 1 \leq i \leq r_{x, y}}}).
\end{multline*}

We compare now Proposition~\ref{prop orbits} with~\cite{OkninskiRenner}*{Theorem~10}. First observe that the authors of~\cite{OkninskiRenner} assume that the lattice $\calI (A)$ is distributive. This assumption is justified by the implication $(3) \implies (4)$ of Theorem~\ref{theoOR}. Using~\cite{Camillo}*{Theorem~1} this is equivalent to $r_{i, j} \leq 1$, for all $i, j$. In other words, it means there are no multiple arrows in $\Delta_A$. Recall, that if there are multiple arrows in $\Delta_A$, then one easily finds a subquiver $Q$ of $\Delta_A$ of type $\tilde{\bbA}_1$ such that $\bh_Q \leq \bw_A |_{Q_0}$.

There are two (equivalent) sets of conditions in~\cite{OkninskiRenner}*{Theorem~10}. We concentrate on the one which is easier to formulate. Namely, we have the following set of conditions:
\begin{enumerate}

\item
there are no cycles in $\Delta_A$;

\item
if $\bw_A (x) \geq 2$, then there at most three arrows starting and at most three arrows terminating at $x$;

\item
if $\bw_A (s \alpha) \geq 2$ and $\bw_A (t \alpha) \geq 2$, for an arrow $\alpha$ of $\Delta_A$, then the number of arrows starting at $s \alpha$ plus the number of arrows terminating at $t \alpha$ is at most $4$.

\end{enumerate}
Obviously, the first condition means that $(Q, \bh_Q) \not \leq (\Delta_A, \bw_A)$, where $Q$ is a quiver of type $\tilde{\bbA}$. In the same way, the second condition excludes the settings $(Q, \bh_Q)$, where $Q$ is a quiver of type $\tilde{\bbD}_4$. Finally, the last condition excludes (up to duality) the settings $(Q, \bd)$, where $Q$ is the quiver
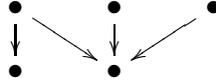
\[
\vcenter{\xymatrix@R=1em{\bullet \ar[d] \ar[rd] & \bullet \ar[d] & \bullet \ar[ld] \\ \bullet & \bullet}}
\]
and $\bd :=
\begin{smallmatrix}
2 & 1 & 1 \\ 1 & 2
\end{smallmatrix}$.
However, the above setting is representation finite. This means that unfortunately~\cite{OkninskiRenner}*{Theorem~10} is false. Note that, for $(Q, \bd)$ as above, $(Q, \bd) \leq (Q', \bh_{Q'})$, where $Q'$ is a Euclidean quiver of one of the types $\tilde{\bbD}_n$, $n > 4$, $\tilde{\bbE}_6$, $\tilde{\bbE}_7$, $\tilde{\bbE}_8$, thus the conditions in~\cite{OkninskiRenner}*{Theorem~10} are sufficient for $A$ to have finitely many $U (A)$-orbits (but not necessary as pointed out above).

\bibsection

\end{document}